\documentclass[10pt]{article}
\usepackage{amsmath,amsthm,amsfonts}
\usepackage{graphicx}
\usepackage{hyperref}
\usepackage[usenames]{color}

\newtheorem{theorem}{Theorem}[section]

\theoremstyle{definition}

\newtheorem{proposition}[theorem]{Proposition}

\newcounter{comcount}

\title{The twisted conjugacy problem for pairs of endomorphisms in  nilpotent groups}

\author{V. ROMAN'KOV\footnote{The first author was partially supported by RFBR grant 07-01-00392.},
E. VENTURA}

\begin{document}

\maketitle

\begin{abstract} An algorithm is constructed that, when given an explicit presentation of a finitely generated nilpotent group $G,$ 
decides for any pair of endomorphisms $\varphi , \psi : G \rightarrow G$ and 
any pair of elements $u, v \in G,$  whether or not 
the equation $(x\varphi )u = v (x\psi )$ has a solution $x \in G.$
Thus it is shown that the problem of the title is decidable. Also we present an algorithm that produces a finite set of generators
of the subgroup (equalizer ) $Eq_{\varphi , \psi }(G) \leq  G$ of all elements $u \in G$ such that $u\varphi = u \psi .$
\end{abstract}

\tableofcontents

\medskip
\begin{quote}{\it \scriptsize 2000 Mathematics Subject Classification.
Primary 20F10. Secondary 20F18; 20E45; 20E36. }

\end{quote}

\section{Introduction}

Let $G$ be a group, and $u, v \in G.$  Given an endomorphism $\varphi \in End(G),$ one says that
$u $ and $v$ are $\varphi -${\it twisted conjugate}, and one writes
$u \sim_{\varphi } v,$ if and only if there exists $x \in G$ such
that $u = (x\varphi )^{-1}vx,$ or equivalently $(x\varphi )u = v
x.$ More generally, given a pair of endomorphisms $\varphi , \psi \in End(G),$ one says that the elements 
$u, v $ are $\varphi , \psi -${\it twisted conjugate}, 
and one writes $u \sim _{\varphi , \psi } v, $
if and only if there exists an element $x \in G$ such that 

\begin{equation}
\label{eq:1} (x \varphi ) u = v (x \psi ).
\end{equation}

The recognition of twisted conjugacy classes with respect to a given
pair of endomorphisms $\varphi , \psi \in End(G)$   in the case of any
finitely generated nilpotent group $G$ is the main concern of
this paper.

\section{Preliminary results}
\label{se:prel}

It is well known (see, for example, \cite{Kar}, \cite{RemRom})
that most of algorithmic problems for finitely generated nilpotent
groups are decidable. In particular, the standard conjugacy problem is decidable in any finitely
generated nilpotent group \cite{Bl}.
Some undecidable problems exist too. The
endomorphism (see \cite{Rom}) and the epimorphism (see \cite{Rem}) problems
are among them, as well as impossibility to decide if a given
equation has a solution  in a given finitely generated nilpotent
group (see  \cite{Rom}).

For the purpose of this paper we need in the following result of
special interest.

\begin{proposition}
\label{pr:1}
 Let $G$ be a finitely generated nilpotent group. Let
$\varphi , \psi \in End(G)$ be a pair of endomorphisms of $G.$
Then there is an algorithm which finds a finite set of
generators of the subgroup (equalizer) 

\end{proposition}

\begin{equation}
\label{eq:2} Eq_{\varphi , \psi } (G) = \{ x \in G | x\varphi  = x
\psi \}.
\end{equation}

\begin{proof} Let $G$ be abelian. Then $Eq_{\varphi , \psi } (G) =
ker (\varphi - \psi ),$ thus there is a standard procedure to find a
generating set for it.

Let $G$ be a finitely generated nilpotent group of class $c+1 \geq 2.$ Suppose by induction 
that there is an algorithm which finds for  any  finitely generated nilpotent
group $H$ of class $\leq c$ and any pair of endomorphisms $\alpha , \beta \in End(H)$ 
a finite set of generators of the equalizer $Eq_{\alpha , \beta }(H).$

Let $C = \gamma _c(G)$ be the last non trivial member of the lower
central series of  $G.$ Then the quotient $H = G/C$ has class $c.$
Since $C$ is invariant for every endomorphism of $G$ we can
consider the induced by $\varphi , \psi \in End(G)$ endomorphisms $\bar{\varphi} , \bar{\psi} \in End(H).$
By the
assumption we can construct a finite set of generators of the equalizer
$Eq_{\bar{\varphi} , \bar{\psi} }(H) \leq H.$ Let $G_1$ be the full preimage
of $Eq_{\bar{\varphi }, \bar{\psi }}(H)$ in $G.$ We call $G_1$  a $C-$equalizer
of $\varphi , \psi , $ and we write $G_1 = Eq_{C,
\varphi , \psi }(G).$ By definition

\begin{equation}
\label{eq:3}
 Eq_{C, \varphi , \psi }(G) = \{g \in G | g\varphi = c_g(g \psi ), \textrm{where} \  
 c_g  \in C \}.
 \end{equation}

\noindent
 Obviously, $Eq_{\varphi , \psi }(G) \leq Eq_{C, \varphi
, \psi }(G),$ and $C \leq Eq_{C, \varphi , \psi }(G).$

Now we define a map
\begin{equation}
\label{eq:4}
 \mu : Eq_{C, \varphi , \psi }(G) \rightarrow C  \   \textrm{by} \   \mu (g)
= c_g.
\end{equation}

Easily to see that this map $\mu $ is a homomorphism, and that the derived
subgroup $(Eq_{C, \varphi , \psi }(G))'$ lies in $ker(\mu ).$

We conclude that $Eq_{\varphi , \psi }(G) = ker(\mu ).$ So, a generating set for $Eq_{\varphi , \psi }(G)$
can be derived by the standard procedure.

\end{proof}

\section{The twisted conjugacy problem for pairs of endomorphisms}
\label{se:intro}

Let $G$ be a finitely generated group, and $\varphi , \psi \in
End(G)$ be any pair of endomorphisms. Let $u \sim_{\varphi , \psi
} v $ be a pair of $\varphi , \psi -$twisted conjugate elements
of $G.$ We write $\{u\}_{\varphi , \psi }$ for the $\varphi , \psi
-$twisted conjugacy class of element $u \in G.$

The question about $\varphi , \psi -$twisted conjugacy of given
elements $u, v \in G$ can be reduced to the case where one of the
elements is trivial. To do this we change $\varphi $ to $\varphi ' =
\varphi \circ \sigma_u,$ where $\sigma_u \in Aut(G)$ is the inner automorphism $ h \mapsto u^{-1}hu.$ 
Hence $(x\varphi )u = v(x\psi )$ if and only if

\begin{equation}
\label{eq:6} x\varphi ' = w (x\psi ),
\end{equation}

\noindent where $w = u^{-1} v.$

Now we are ready to prove our main result about twisted conjugacy
in finitely generated nilpotent groups.

\medskip
\begin{theorem}
\label{th} Let $G$ be a finitely generated nilpotent group of
class $c \geq 1.$ Then there exists an algorithm which  decides
the twisted conjugacy problem for any pair of
endomorphisms $\varphi , \psi \in End(G).$
\end{theorem}

\begin{proof} Induction by $c.$ For $c = 1$ (abelian case) the
statement is obviously true.

Suppose that the statement is true in the case of any finitely generated nilpotent group $N$ of
class $\leq c - 1.$ 

Let $u, v \in G$ be any pair of elements, and $\varphi , \psi \in
End(G)$ be any pair of endomorphisms. We change $\varphi $ to $\varphi ' = \varphi \circ \sigma_u, $ and write the equation (\ref{eq:6}) with $w = u^{-1}v$, as were 
explained above. Since the last non trivial
member $C$ of the lower central series  of $G$ is invariant for
every endomorphism,  we decide the twisted conjugacy  problem in $G/C$ with respect to the 
induced by $\varphi ', \psi $ endomorphisms $\bar{\varphi} , \bar{\psi } \in End(G/C)$ and the induced by $w$ element $\bar{w} \in G/C.$. 
More exactly, we decide if there exists an element $\bar{x} \in G/C$ for which

\begin{equation}
\label{eq:8} \bar{x} \bar{\varphi }' = \bar{w} (\bar{x} \bar{\psi} ).
\end{equation}

By our assumption we can decide this problem effectively. 
If such element $\bar{x}$ does not exist the element $x$ does not exist too.

Suppose that $\bar{x}$ exists. Then there is an element $x_1
\in G$ for which

\begin{equation}
\label{eq:9} x_1 \varphi ' = c g (x_1 \psi ),
\end{equation}

\noindent where $c \in C.$ If $x_2 \varphi '  = c' g (x_2 \psi )$
for some element $x_2 \in G$ and $c' \in C,$  we derive that

\begin{equation}
\label{eq:10} (x_2^{-1}x_1)\varphi ' = c'' ((x_2^{-1}x_1)\psi ),
\end{equation}

\noindent
where $c'' = (c')^{-1} c \in  C.$ Thus $x_2^{-1}x_1 \in
Eq_{C, \varphi ', \psi  }(G).$ 
In the case when $x_2 = x$ is a solution of (\ref{eq:6}) we have $c'' = c.$
Conversely, the equality $c'' = c$ means that there is a solution $x$ of (\ref{eq:6}).

By Proposition \ref{pr:1} we construct a finite generating set of
$Eq_{\bar{\varphi }', \bar{\psi }}(G/C),$ and so  we can construct a finite
generating set of its full preimage:

\begin{equation}
\label{eq:11} Eq_{C, \varphi , \psi } (G) = gp(g_1, ..., g_l).
\end{equation}

Thus we have

\begin{equation}
\label{eq:12} g_i \varphi = c_i (g_i \psi ),
\end{equation}

\noindent where $c_i \in C \  \textrm{for} \  i = 1, ..., l.$

Then we apply a the homomorphism $\mu $ defined by the map $g_i \mapsto c_i$ for $i = 1, ..., l.$.
As we noted above, $1 \sim_{\varphi , \psi } w$ if and only if the element
$c$ belongs to the image $ (Eq_{C, \varphi ', \psi }(G))\mu = gp(c_1, ..., c_l).$
So, our problem is reduced to the membership problem in a finitely
generated abelian group $C.$ It is well known that this problem
is decidable.
\end{proof}

\end{document}